\documentclass[10pt]{article}
\usepackage[english]{babel}
\usepackage{amsmath,amsthm}
\usepackage{amsfonts}
\newtheorem{theorem}{Theorem}[section]
\newtheorem{corollary}[theorem]{Corollary}

\newtheorem{proposition}[theorem]{Proposition}
\newtheorem{definition}[theorem]{Definition}
\newtheorem{remark}[theorem]{Remark}
\newtheorem{example}[theorem]{Example}
\numberwithin{equation}{section}

\begin{document}
\title{Symmetries, psudosymmetries and conservation laws in Lagrangian and Hamiltonian $k$-symplectic formalisms}
\author{Florian Munteanu}
\date{Department of  Applied Mathematics, University of
Craiova\\Al.I. Cuza 13, Craiova 200585, Dolj, Romania\\
munteanufm@central.ucv.ro}
\maketitle
\begin{abstract}
In this paper we will present Lagrangian and Hamiltonian $k$-symplectic formalisms, we will recall the notions of symmetry and conservation law and we will define the notion of  pseudosymmetry as a natural extension of symmetry. Using symmetries and pseudosymmetries, without the help of a Noether type theorem,  we will obtain new kinds of conservation laws for $k$-symplectic Hamiltonian systems and $k$-symplectic Lagrangian systems.
\end{abstract}

{\bf AMS Subject Classification (2010):} 70S05, 70S10, 53D05.

{\bf Key words:} symmetry, pseudosymmetry, conservation law, Noether theorem, $k$-symplectic
Hamiltonian system, $k$-symplectic Lagrangian system.

\section{Introduction}
The $k$-symplectic formalism (\cite{aw1}, \cite{aw2}, \cite{gun}, \cite{muntsalg}) is the generalization to field theories of the standard
symplectic formalism in autonomous Mechanics, which is the geometric
framework for describing autonomous dynamical systems (\cite{abhmar}, \cite{arnold}). This formalism is based on the polysymplectic formalism developed by G\"{u}nther (\cite{gun}). The $k$-Symplectic Geometry provides the simplest geometric framework for describing certain class of first-order classical field theories. Using this description we analyze different kinds of symmetries for the Hamiltonian and Lagrangian formalisms of these field theories, including the study of conservation laws associated to them and stating Noether's Theorem (\cite{rrsv}, \cite{modestobuc}, \cite{modestoconfer}). Further more, we will generalize  the study of symmetries and conservation laws from classical (symplectic, $k=1$)
  formalism  to the $k$-symplectic formalism  for obtain new kinds of conservation laws for $k$-symplectic Hamiltonian and Lagrangian systems. A similar study for the case
of higher order tangent bundles geometry was done by the author in
\cite{muntfl3}, \cite{muntfl}.

In this paper we will revisited the study of symmetries, conservation laws and relationship between this in the framework of $k$-symplectic
geometry and we will improved the results obtained in \cite{muntfliasi} and \cite{munt2013}. More exactly, we intend to extend the study of symmetries and
conservation laws from Classical Mechanics to the first-order
classical field theories, both for the Lagrangian and Hamiltonian
formalisms, using G\"{u}nther's $k$-symplectic description, and
considering only the regular case. We will find new kinds of
conservation laws, nonclassical, without the help of a Noether's
type theorem, using only the  relationship between symmetries, pseudosymmetries and conservation laws.

The study of symmetries and conservation laws for
$k$-symplectic Hamiltonian systems is, like in the classical case,
a topic of great interest and was developped recently by M.
Salgado, N. Roman-Roy, S. Vilarino in \cite{rrsv}, \cite{modestoconfer} and  L. Bua, I. Buc\u{a}taru, M. Salgado in \cite{modestobuc}. Further more, in the paper \cite{mrrsv} J.C. Marrero, N. Roman-Roy, M. Salgado, S. Vilarino begin the study of symmetries and conservation laws for $k$-cosymplectic Hamiltonian systems, like an extension to field theories of the standard cosymplectic formalism for nonautonomous mechanics (\cite{mmopm}, \cite{mmopmm}). In \cite{rrsv} the Noether's theorem, obtained for a $k$-symplectic Hamiltonian
system, associates conservation laws to so-called Cartan
symmetries. However, these kinds of symmetries do not exhaust the
set of symmetries. As is known, in mechanics and physics there are
symmetries which are not of Cartan type, and which generate
conserved quantities, i.e. conservation laws (\cite{lmr}, \cite{r1}, \cite{r2}).

In the second section are presented the basic tools  and, also the classical result who will be generalized in the
last section. Following \cite{rrsv}, \cite{modestoconfer}, in sections three and four we will review the essential geometric elements of $k$-symplectic formalism who need us to explain and to obtain the results from the last section. So, by generalization from symplectic geometry to $k$-symplectic
geometry, we will obtain new kinds of conservation laws for
$k$-symplectic Hamiltonian and Lagrangian systems, without the help of a Noether
type theorem and without the use of a variational principle. The
main result is a generalization from the classical case ($k=1$) of
a results of G.L. Jones (\cite{jones}) and M.
Cr\^{a}\c{s}m\u{a}reanu (\cite{crasmareanu}).

All manifolds and maps are $C^{\infty}$. Sum over crossed repeated indices is understood.
\section{Basic tools}
Let $M$ be a $n$-dimensional manifold, $C^\infty (M)$ the
ring of real-valued functions, $\mathcal{X}(M)$ the Lie
algebra of vector fields and $A^p(M)$ the $C^\infty (M)$-module of $p$-forms, $
1\leq p\leq n$.

Let us recall that if $\Delta$ is a distribution with a constant rank $k$ on $M$  and $\Phi : M\rightarrow M$ is a diffeomorfism  of $M$, then  $\Phi$ is called an {\it invariant transformation} or {\it finite symmetry} of $\Delta$ if for all $x \in M$, $T \Phi (\Delta_x) \subset \Delta_{\Phi (x)}$ (\cite{krupkova}).

If $\{ \Phi_t^{\xi} \}_{t}$ denote the local one-parameter group of transformations of the vector field $\xi$ on $M$, then $\xi$ is called {\it symmetry} or {\it infinitesimal symmetry} or {\it dynamical symmetry}  of $\Delta$ if for all $t$, $\{ \Phi_t^{\xi} \}_{t}$ is an invariant transformation of $\Delta$.

$\xi$ is a symmetry of $\Delta$ if and only if for all $\zeta \in \Delta$, $ [\xi,\zeta] \in \Delta$, or equivalently, the local flow of $\xi$ transfer integral mappings in integral mappings and consequently, for any integral manifold $Q$ of $\Delta$, $\{ \Phi_t^{\xi} \}_{t} (Q)$ is another integral manifold of $\Delta$.

A function $g:U\rightarrow\mathbf{R}$ ($U$ being an open subset of $M$) is called {\it first integral} or {\it conservation law} of $\Delta$ if the one-form $dg$ belongs to $\Delta$, i.e. $i_{\xi} dg = 0$, for all $\xi \in \Delta$.

If $g$ is a first integral of $\Delta$ on $U$ and $Q$ is an integral manifold of $\Delta$ with integral mapping $i:Q\rightarrow U\subseteq M$, then $d(g \circ i)=0$, that means the function $g$ is constant along the integral manifold $Q$.

For $X\in \mathcal{X}(M)$ with local expression $
X=X^i(x)\frac \partial {\partial x^i}$ we consider the system of
ordinary differential equations which give the flow $\{\Phi
_t\}_t$ of $X$, locally,
\begin{equation}
\dot{x}^i (t)=\frac{dx^i}{dt}(t)=X^i(x^1(t),\ldots ,x^n(t)).
\label{d2}
\end{equation}
\textit{A dynamical system} is a couple $(M,X)$. A dynamical system is
denoted by the flow of $X$, $\{\Phi _t\}_t$ or by the system of
differential equations (\ref{d2}).

A function $f\in C^\infty (M)$ is called \textit{conservation law}
for dynamical system $(M,X)$ if $f$ is constant along the every
integral curves of $X$ (solutions of (\ref{d2})), that is
\begin{equation}
L_Xf=0,  \label{d5}
\end{equation}
where $L_Xf$ means the Lie derivative of $f$\ with respect to $X$.

If $Z\in \mathcal{X}(M)$ is fixed, then $Y\in \mathcal{X}(M)$ is called $Z$-
\textit{pseudosymmetry} for $(M,X)$ if there exists $f$ $\in
C^\infty (M)$ such that $L_XY=fZ$. A $X$-pseudosymmetry for $X$ is
called \textit{pseudosymmetry} for $(M,X)$. $Y\in \mathcal{X}(M)$
is called \textit{symmetry} for $(M,X)$ if $L_XY=0$. Recall that $\omega \in A^p(M)$ is called \textit{invariant
form} for $(M,X)$ if $L_X\omega =0$.

\begin{example} \rm  (\cite{crasmareanu2}, \cite{ibanez}) For Nahm's system from
the theory of static SU(2)-monopoles:
\begin{equation}\label{nahmeq}
\frac{dx^1}{dt} = x^2 x^3 ,\, \,  \frac{dx^2}{dt} = x^3 x^1 , \,
\, \frac{dx^3}{dt} = x^1 x^2  \, ,
\end{equation}
the vector field $Y= x^1 \frac{\partial}{\partial x^1}+ x^2 \frac{\partial}{\partial x^2} + x^3 \frac{\partial}{\partial x^3}$ is a pseudosymmetry.
\end{example}
The notion of pseudosymmetry defined above is a weaker notion of symmetry. This is a natural generalization of the notion of symmetry for
a system of ordinary differential equations  (\ref{d2}). Symmetries and pseudosymmetries are just infinitesimal symmetries of the distribution generated by the vector field $X$ (\cite{krupkova}).

The next theorem which gives the association between
pseudosymmetries and conservation laws is due to M. Cr\^{a}\c{s}m\u{a}reanu (\cite{crasmareanu})
and G.L. Jones (\cite{jones}). We will
find new kinds of conservation laws, nonclassical, without the
help of Noether's type theorem.
\begin{theorem}\label{thmgen}
 Let $X\in \mathcal{X}(M)$ be a fixed vector field and
$\omega \in A^p(M)$ be a invariant $p$-form for $X$. If $Y\in $
$\mathcal{X} (M)$ is
symmetry for $X $ and $S_1$, $\ldots $, $S_{p-1}\in \mathcal{X}(M)$ are $%
(p-1)$ $Y$-pseudosymmetry for $X$ then
\begin{equation}
\Phi =\omega (X,S_1,\ldots ,S_{p-1})  \label{d9}
\end{equation}
or, locally,
\begin{equation}
\Phi = S^{i_1}_1 \cdots S^{i_{p-1}}_{p-1} Y^{i_p} \omega_{i_{1}
\ldots i_{p-1} i_{p}}
\end{equation}
is a conservation laws for $(M,X)$.

Particularly, if $Y$, $S_1$, $\ldots $, $S_{p-1}$ are symmetries
for $X$ then $\Phi $ given by (\ref{d9}) is conservation laws for
$(M,X)$.
\end{theorem}
If $(M,\omega )$ is a symplectic manifold then the dynamical system $(M,X)$ is said to be a {\em dynamical Hamiltonian system} if there exists a function $H\in C^\infty (M)$ (called {\em the Hamiltonian}) such that $i_X\omega =-dH$, where $i_X$ denotes the interior product with respect to $X$.

It is known that the symplectic form $\omega $ is an invariant
2-form for $(M,X)$ and the Hamiltonian $H$ is a conservation law
for $(M,X)$.

Now, we can apply Theorem \ref{thmgen} to the dynamical Hamiltonian systems.
\begin{proposition}
Let be $(M,X_{H})$ a Hamiltonian system on the symplectic manifold $%
(M,\omega)$, with the local coordinates $(x^i,p_i)$. If $Y\in \mathcal{X}(M)$
is a symmetry for $X_{H}$ and $Z\in \mathcal{X}(M)$ is a $Y$-pseudosymmetry
for $X_{H}$ then
\begin{equation}
\Phi =\omega (Y,Z)  \label{hsdtt}
\end{equation}
is a conservation law for the Hamiltonian system $(M,X_{H})$.

Particularly, if $Y$ and $Z$ are symmetries for $X_{H}$ then $\Phi $ from (\ref{hsdtt}) is a conservation law for $(M,X_{H})$.
\end{proposition}
\section{$k$-Symplectic Hamiltonian formalism}
Let $(T_k^1)^{\ast} M = T^{\ast} M \bigoplus \ldots \bigoplus T^{\ast} M $ (the Whitney sum of $k$ copies of $T^{\ast} M$) be the $k$-{\it cotangent bundle} of a $n$-dimensional differentiable manifold $M$, with the projection $\tau^{\ast} : (T_k^1)^{\ast} M \rightarrow M$. The natural coordinates on $(T_k^1)^{\ast} M$ are $(x^i, p_i^A)$, $1 \leq i \leq n$, $1 \leq A \leq k$.

The {\it canonical $k$-symplectic structure} in $(T_k^1)^{\ast} M$ is $(\omega_A,V)$, where $V=\ker (\tau^{\ast})_{\ast}$ and $\omega_A = (\tau_A^{\ast})^{\ast} \omega = - d (\tau_A^{\ast})^{\ast} \theta = -d \theta_A$. $\omega = - d \theta$ is the canonical symplectic structure in $T^{\ast} M$, $\theta$  is the Liouvile $1$-form in $T^{\ast} M$ and $\tau_A^{\ast} : (T_k^1)^{\ast} M \rightarrow T^{\ast} M$ is the projection on the $A^{th}$-copy  $T^{\ast} M$ of $(T_k^1)^{\ast} M$. Locally, $ \omega_A = - d \theta_A = - d (p_i^A dx^i) = dx^i \wedge d p_i^A \,$.

If $Z \in \mathcal{X} (M)$ has the local $1$-parametric group $h_s : Q \rightarrow Q$, then the {\it canonical lift} of $Z$ to $(T_k^1)^{\ast}_x M$ is the vector field $Z^{C\ast} \in \mathcal{X}((T_k^1)^{\ast} M)$ whose local $1$-parametric group is $(T_k^1)^{\ast} (h_s) : (T_{k}^{1})^{\ast} M \rightarrow (T_{k}^{1})^{\ast} M $. Locally, if $Z=Z^i\frac{\partial}{\partial x^i} \,$, then
$Z^{C\ast} = Z^i\frac{\partial}{\partial x^i} - p^A_j \frac{\partial Z^j}{\partial x^k} \frac{\partial}{\partial p_k^A} \,$.
\begin{definition} \rm
Let $T_k^1 \mathcal{M} = T \mathcal{M} \bigoplus \ldots \bigoplus T \mathcal{M} $ be the $k$-tangent bundle  of a manifold $\mathcal{M}$.\\
1) {\em A $k$-vector field} on $\mathcal{M}$ is a section $\mathbf{X} : \mathcal{M} \rightarrow T_k^1 \mathcal{M}$ of $\tau : T_k^1 \mathcal{M} \rightarrow \mathcal{M}$, the natural projection. A $k$-vector field $\mathbf{X}$ on $\mathcal{M}$ defines a family of vector fields $X_1$, ..., $X_k$ on $\mathcal{M}$ by $X_A = \tau_A \circ \mathbf{X}$, where $\tau_A : T_k^1 \mathcal{M} \rightarrow T \mathcal{M}$ is the projection on the $A^{th}$-copy  $T \mathcal{M}$ of $T_k^1 \mathcal{M}$.\\
2) {\em An integral section} of $\mathbf{X}$ at a point $x \in \mathcal{M}$ is a map $\psi : U_0 \subset \mathbf{R}^k \rightarrow \mathcal{M}$, with $0 \in U_0$, such that $\phi (0) = x$, $\psi_{\ast} (t) \left( \frac{\partial}{\partial t^A} (t) \right) = X_A (\psi (t))$, for every $t \in U_0$.\\
3) A $k$-vector field $\mathbf{X}$ is called {\em integrable} if there is an integral section at every point of $\mathcal{M}$.
\end{definition}
Let $H : (T_k^1)^{\ast} M \rightarrow \mathbf{R}$ be a {\it Hamiltonian function}. The family $\left( (T_{k}^{1})^{\ast} M,\omega_{A},H\right) $ is  a $k$-{\it symplectic Hamiltonian system}. The {\it Hamilton-de Donder-Weyl (HDW) equations} are
\begin{equation} \label{hdw}
\frac{\partial H}{\partial x^i} (\psi (t))  = - \sum\limits_{A=1}^{k} \frac{\partial \psi_i^A}{\partial t^A} (t) \, , \, \,
\frac{\partial H}{\partial p_i^A} (\psi (t))  = \frac{\partial \psi^i}{\partial t^A} (t)  \, ,
\end{equation}
where $\psi : \mathbf{R}^k \rightarrow (T_k^1)^{\ast} M$, $\psi (t) = \left( \psi^i (t) , \psi^A_i (t)  \right)$, is  a solution.

We denote by $\mathcal{X}_{H}^{k} \left( (T_{k}^{1})^{\ast} M \right)$ the set of $k$-vector fields on $(T_{k}^{1})^{\ast} M$ solutions to
\begin{equation}\label{hdwplus}
\sum\limits_{A=1}^{k} i_{X_A} \omega_A = dH \, .
\end{equation}
Any $k$-vector field $(X_1,\dots,X_k)$ which is a solution of (\ref{hdwplus}) will be called an \textit{evolution $k$-vector field} associated with the Hamiltonian function $H$. It should be noticed that in general the solution to the above equation is
not unique. Nevertheless, it can be proved \cite{mmopm} that there always
exists an \textit{evolution $k$-vector field} associated with a Hamiltonian
function $H$.

Then, if $\mathbf{X} \in \mathcal{X}_{H}^{k} \left( (T_{k}^{1})^{\ast} M \right)$ is integrable and $\psi : \mathbf{R}^k \rightarrow (T_k^1)^{\ast} M$ is an integral section of $\mathbf{X}$, then $\psi (t) = \left( \psi^i (t) , \psi^A_i (t)  \right)$ is  a solution to the HDW equations (\ref{hdw}).

In  \cite{modestobuc}, \cite{modestoconfer}, \cite{rrsv} it is introduced next definition for a conservation law of the Hamilton-de Donder-Weyl (HDW) equations (\ref{hdw}) on $(T_{k}^{1})^{\ast} M$:
\begin{definition} \rm (\cite{modestoconfer}, \cite{rrsv})  A map  $\mathbf{\Phi }=\left( \Phi _{1},\ldots ,\Phi _{k}\right) : (T_{k}^{1})^{\ast}M \longrightarrow \mathbf{R}^{k}$ is called \textbf{conservation law} for the Hamilton-de Donder-Weyl (HDW) equations (\ref{hdw}) if the divergence of the function
$$\mathbf{\Phi } \circ  \psi = \left( \Phi _{1} \circ  \psi,\ldots ,\Phi _{k}\circ  \psi \right) : U\subset \mathbf{R}^{k}\rightarrow\mathbf{R}^{k}$$
is zero, for every $\psi  : U \subset \mathbf{R}^{k} \rightarrow M $ solution of the Hamilton-de Donder-Weyl (HDW) equations (\ref{hdw}), that is
\begin{equation}
\sum\limits_{A=1}^{k} \frac{\partial \left( \Phi_A \circ  \psi \right) }{\partial t^A} (t) = 0 \, .
\label{lagrconslaw}
\end{equation}
\end{definition}
In \cite{rrsv} was proved the next result:
\begin{proposition} \label{phcons1}
 If $\mathbf{\Phi }=\left( \Phi _{1},\ldots ,\Phi _{k}\right) : (T_{k}^{1})^{\ast} M \longrightarrow \mathbf{R}^{k}$ is a conservation law for the Hamilton-de Donder-Weyl (HDW) equations (\ref{hdw}), then for every integrable $k$-vector field $\mathbf{X}=(X_{1},\dots ,X_{k}) \in \mathcal{X}_{H}^{k}((T_{k}^{1})^{\ast} M)$ we have
\begin{equation}
\sum\limits_{A=1}^{k} L_{X_A} \Phi_A = 0 \, .
\end{equation}
\end{proposition}
The converse of  proposition \ref{phcons1} may not be true and the reason is that we might have solutions $\psi$ of the Hamilton-de Donder-Weyl (HDW) equations (\ref{hdw}) that are not solutions to some $\mathbf{X} \in \mathcal{X}_{H}^{k}((T_{k}^{1})^{\ast} M)$.

However, under some assumption, the converse is true (\cite{modestobuc}):
\begin{proposition} \label{phcons2}
If we assume that there exists a vector field $X \in \mathcal{X} ((T_{k}^{1})^{\ast} M)$ such that
\begin{equation}
i_X (\omega_0)_A = d f_A, \forall 1 \leq A \leq k
\end{equation}
for some functions $f_A : (T_{k}^{1})^{\ast} M \longrightarrow \mathbf{R}^{k}$, then $\mathbf{F}=(f_1, \ldots ,f_k)$ is a conservation law of HDW equations (\ref{hdw}) if and only if $\sum\limits_{A=1}^{k} X_A (f_A) = 0$, for every integrable $k$-vector field $\mathbf{X}=(X_{1},\dots ,X_{k}) \in \mathcal{X}_{H}^{k}((T_{k}^{1})^{\ast} M)$.
\end{proposition}
\begin{definition} \rm \label{defhamsym}
Let $\left( (T_{k}^{1})^{\ast} M,\omega_{A},H\right) $ be a $k$-symplectic Hamiltonian system.\\
1) A {\it symmetry} is a diffeomorphism $\Phi : (T_{k}^{1})^{\ast} M \rightarrow (T_{k}^{1})^{\ast} M$ such that, for every solution $\psi$ of the Hamilton-de Donder-Weyl (HDW) equations (\ref{hdw}), we have that $\Phi \circ \psi$ is also a solution Hamilton-de Donder-Weyl (HDW) equations (\ref{hdw}).\\
2) An {\it infinitesimal symmetry} is a vector field $Y \in \mathcal{X} ((T_{k}^{1})^{\ast} M)$ whose local flows are local symmetries.\\
3) A {\it Cartan symmetry} is a diffeomeorphism $\Phi : (T_{k}^{1})^{\ast} M \rightarrow (T_{k}^{1})^{\ast} M$ such that $\Phi^{\ast} \omega_A = \omega_A$, for all $1 \leq A \leq k$, and $\Phi^{\ast} H=H$ (up to a constant).\\
4) An {\it infinitesimal Cartan symmetry} is a vector field $Y \in \mathcal{X} ((T_{k}^{1})^{\ast} M)$ such that\\
$L_Y \omega_A$, for all $1 \leq A \leq k$, and $L_Y H = 0$.
\end{definition}

If $\Phi = (T_k^1)^{\ast} \phi$, for some diffeomorphism $\phi : M \rightarrow M$, then the (Cartan) symmetry  $\Phi$ is said to be {\it natural}.

If $Y=Z^{C\ast}$ for $Z \in \mathcal{M}$, then the infinitesimal (Cartan) symmetry is said to be {\it natural}.

If $\Phi$ is a Cartan symmetry then it is a symmetry.

If $\mathbf{X}=(X_{1},\dots ,X_{k}) \in \mathcal{X}_{H}^{k}((T_{k}^{1})^{\ast} M)$, then  $\Phi_{\ast} \mathbf{X}=(\Phi_{\ast} X_{1},\dots ,\Phi_{\ast} X_{k}) \in \mathcal{X}_{H}^{k}((T_{k}^{1})^{\ast} M)$.
\begin{proposition} (\cite{rrsv}, \cite{modestoconfer})
Let $Y \in \mathcal{X} \left( (T_{k}^{1})^{\ast} M \right)$ be an infinitesimal Cartan symmetry. Then, for every $p \in (T_{k}^{1})^{\ast} M$, there is an open neighbourhood $U_p$, such that\\
i) There exists $f_A \in C^{\infty} (U_p)$, unique up to a constant functions, such that $i_Y \omega_A = df_A$.\\
ii) There exists $\zeta_A \in C^{\infty} (U_p)$, verifying $L_Y \theta_A = d \zeta_A$ on $U_p$ and then $f_A = i_Y \theta_A - \zeta_A$ (up to a constant function on $U_p$).
\end{proposition}
\begin{theorem} (Noether Theorem) (\cite{rrsv}, \cite{modestoconfer}) Let $Y \in \mathcal{X} \left( (T_{k}^{1})^{\ast} M \right)$ be an infinitesimal Cartan symmetry.\\
i) For every $p \in (T_{k}^{1})^{\ast} M$, there is an open neighbourhood $U_p$, such that the functions $f_A = i_Y \theta_A - \zeta_A$, $1 \leq A \leq k$, define a conservation law $\mathbf{f}=(f_1, \ldots , f_k)$ on $U_p$.\\
ii) For every $\mathbf{X}=(X_{1},\dots ,X_{k}) \in \mathcal{X}_{H}^{k}\left( (T_{k}^{1})^{\ast} M \right)$, we have $\sum\limits_{A=1}^k L_{X_A} f_A = 0$ on $U_p$.
\end{theorem}
\section{$k$-Symplectic Lagrangian formalism}
Let be $T_k^1 M = T M \bigoplus \ldots \bigoplus T M$ the $k$-{\it tangent bundle} of a manifold $M$, with natural projection $\tau : T_k^1 M \rightarrow M$. The natural coordinates on $T_k^1 M$ are $(x^i,v^i_A) \,$.

Locally, for $Z_x = a^i \frac{\partial}{\partial x^i} \in T_x M$ the {\it vertical $A$-lift} of $Z$ at $\left( v_{1_x}, \ldots, v_{k_x} \right) \in T_k^1 M$ is the vector field $(Z_x)^{V_A} \left( v_{1_x}, \ldots, v_{k_x} \right) = a^i \frac{\partial}{\partial v_A^i}\mid_{\left( v_{1_x}, \ldots, v_{k_x} \right)} \,$.

Locally, the {\it Liouville vector field} is $\Delta  = \sum\limits_{A=1}^{k} v^i_A \frac{\partial}{\partial v_A^i}$ and the $k$-{\it tangent structure} on $T_k^1 M$ is the set $(S^1, \ldots , S^k)$ of $(1,1)$-tensor fields defined by  $S^A = \frac{\partial}{\partial v_A^i} \otimes dx^i \,$.

Let $Z \in \mathcal{X} (M)$ with the local $1$-parametric group $h_s : Q \rightarrow Q$. Then the {\it canonical lift} of $Z$ to $(T_k^1)_x M$ is the vector field $Z^{C} \in \mathcal{X}(T_k^1 M)$ whose local $1$-parametric group is $ T_k^1 h_s : T_k^1 M \rightarrow T_k^1 M  \,$.
 Locally, if $Z=Z^i\frac{\partial}{\partial x^i} \,$, then
$Z^{C} = Z^i\frac{\partial}{\partial x^i} + v_A^j \frac{\partial Z^k}{\partial x^j} \frac{\partial}{\partial v_A^k} \,$.
\begin{definition} \rm
A second order partial differential equation (SOPDE) is a $k$-vector field $\Gamma$ in $T_k^1 M $ which is a section of the projection
  $T_k^1 \tau : T_k^1 (T_k^1 M ) \rightarrow T_k^1 M) \,$.
\end{definition}
Locally,  a SOPDE $\Gamma = (\Gamma_1, \ldots , \Gamma_k)$ is given by the vector fields $\Gamma_A = v^i_A \frac{\partial}{\partial x^i}+ \left( \Gamma_A \right)^{i}_{B} \frac{\partial}{\partial v_B^i}  \,$.
\begin{proposition}
If $\psi$ is an integral section of an integrable SOPDE $\Gamma$, then $\psi = \phi^{(1)}$, where $\phi^{(1)}$ is the first prolongation of $\phi = \tau \circ \psi $ and $\phi$ is a solution to the system
\begin{equation}\label{sopde}
\frac{\partial^2 \phi^i}{\partial t^A \partial t^B} (t)= \left( \Gamma_A \right)^{i}_{B} \left( \phi^j (t) , \frac{\partial \phi^j}{\partial t^C} (t)  \right)  \, .
\end{equation}
Conversely, if $\phi : \mathbf{R}^k \rightarrow M $ is a solution of the system (\ref{sopde}), then $\phi^{(1)}$ is an integral section of SOPDE $\Gamma$.
\end{proposition}
Let $L : T_k^1 M \rightarrow \mathbf{R}$ be a Lagrangian. The {\em  Euler-Lagrange equations} for $L$ are
\begin{equation} \label{eulerlagrangeeq}
\sum\limits_{A=1}^{k} \frac{\partial}{\partial t^A} (t)  \left( \frac{\partial L}{\partial v^i_A} ( \phi (t) ) \right) =
 \frac{\partial L}{\partial x^i} (\psi (t))  \, , \, \,
 v^i_A (\psi (t)) \frac{\partial \psi^i}{\partial t^A} (t)   \, ,
\end{equation}
whose solutions are maps $\psi : \mathbf{R}^k \rightarrow T_k^1 M$. We observe that  $\psi (t) =  \phi^{(1)} (t)$ for $\phi = \tau \circ \psi $.

We can introduce the forms associated to $L$, $(\theta_L)_A = dL \circ S^A \in \Omega^1 (T_k^1 M)$, $(\omega_L)_A = - d (\theta_L)_A \in \Omega^2 (T_k^1 M)$, and the {\it energy Lagrangian function} $E_L = \Delta (L) - L \in C^{\infty} (T_k^1 M)$. Locally,
$$  (\theta_A)_L = \frac{\partial L}{\partial v_A^i} dx^i \, , \  \, (\omega_A)_L = \frac{\partial^2 L}{\partial x^j \partial v_A^i} dx^i \wedge dx^j +
\frac{\partial^2 L}{\partial v_B^j \partial v_A^i} dx^i \wedge dv^j_B \, , \  \,  E_L = v_A^i \frac{\partial L}{\partial v_A^i} - L  \, .$$
The Lagrangian $L$ is {\it regular} if the matrix $\left( \frac{\partial^2 L}{\partial v^i_A \partial v^j_B} \right)$ is regular at every point of $T^1_k M$.\\
This is equivalent to say that $((\omega_L)_1, \ldots , (\omega_L)_k; V)$, $V=\ker \tau_{\ast}$, is a $k$-symplectic structure.\\
The family $\left( T^1_k M , (\omega_L)_A , E_L \right)$ is called a $k$-{\it symplectic Lagrangian system}.

Let $\mathcal{X}^k_L (T^1_k M)$ be the set of $k$-vector fields $\Gamma = (\Gamma_1 , \ldots , \Gamma_k)$ in $T^1_k M$ solutions to
\begin{equation}\label{hdwplusl}
\sum\limits_{A=1}^{k} i_{\Gamma_A} (\omega_L)_A = d E_L \, .
\end{equation}
Locally, if $\Gamma_A = (\Gamma_A)^i \frac{\partial}{\partial x^i}+ \left( \Gamma_A \right)^{i}_{B} \frac{\partial}{\partial v_B^i} $ and $L$ is regular, then $\Gamma$ is a solution of (\ref{hdwplusl}) if and only if
$$ \frac{\partial^2 L}{\partial x^j \partial v^i_A} v_A^j + \frac{\partial^2 L}{\partial v^i_A \partial v^j_B}  \left( \Gamma_A \right)^{j}_{B}=\frac{\partial L}{\partial x^i} \, , \, \  (\Gamma_A)^i = v_A^i \, .$$
Thus, if $\Gamma \in \mathcal{X}^k_L (T^1_k M)$, then it is a SOPDE and, if it is integrable, its integral sections are first prolongations of maps $\phi : \mathbf{R}^k \rightarrow M$ solution to the Euler-Lagrange equations (\ref{eulerlagrangeeq}).

Now, if we consider a regular Lagrangian on $T_k^1 M$, then we can rewrite the results from above for the Hamiltonian system $\left( T_{k}^{1}M,(\omega_{L})_{A},E_{L}\right) $ with Hamiltonian function $H=E_L=v^i_A \frac{\partial L}{\partial v^i_A} - L$.

In the classical case ($k=1$), let us recall that Cartan
symmetries induce and are induced by constants of motions
(conservation laws), and these results are known as Noether
Theorem and its converse (\cite{crampin}, \cite{crasmareanu},  \cite{jones}, \cite{noether}, \cite{olver}).

For the higher order case the problem was be solved by L. Bua,
I.Buc\u{a}taru and M. Salgado in \cite{modestobuc}. So, for $k >
1$  the Noether Theorem is also true, that is each Cartan symmetry
induces a conservation law (defined for a regular Lagrangian on
$T_k^1 M$, like in \cite{modestobuc}). However, the converse of
Noether Theorem may not be true and in \cite{modestobuc} is
provided some examples of conservation laws that are not induced
by Cartan symmetries.

In  \cite{modestobuc}, \cite{rrsv} it is introduced next definition for a conservation law of Euler-Lagrange equations (\ref{eulerlagrangeeq}) on $T^1_k M$:
\begin{definition} \rm (\cite{modestobuc})  A map  $\mathbf{\Phi }=\left( \Phi _{1},\ldots ,\Phi _{k}\right) : T^1_k M \longrightarrow \mathbf{R}^{k}$ is called \textbf{conservation law} for the Euler-Lagrange equations (\ref{eulerlagrangeeq}) if the divergence of
$$\mathbf{\Phi } \circ  \phi^{(1)} = \left( \Phi _{1} \circ  \phi^{(1)},\ldots ,\Phi _{k}\circ  \phi^{(1)} \right) : U\subset \mathbf{R}^{k}\rightarrow\mathbf{R}^{k}$$
is zero, for every $\phi  : U \subset \mathbf{R}^{k} \rightarrow M $ solutions of the Euler-Lagrange equations (\ref{eulerlagrangeeq}), that is
\begin{equation}
\sum\limits_{A=1}^{k} \frac{\partial \left( \Phi_A \circ  \phi^{(1)} \right) }{\partial t^A} (t) = 0 \, .
\label{lagrconslaw}
\end{equation}
\end{definition}
In \cite{modestobuc} was proved the next results:
\begin{proposition} \label{pcons1}
Let $\mathbf{\Phi}=\left( \Phi _{1},\ldots ,\Phi _{k}\right) : T^1_k M \longrightarrow \mathbf{R}^{k}$ be a conservation law for the Euler-Lagrange equations (\ref{eulerlagrangeeq}). If $\xi=(\xi_{1},\dots ,\xi_{k})$ is an integrable SOPDE which belongs to $\mathcal{X}_{L}^{k}(T_{k}^{1}M)$, then
\begin{equation}
\sum\limits_{A=1}^{k} L_{\xi_A} \Phi_A = 0 \, . \label{defcons2}
\end{equation}
\end{proposition}
The converse of  proposition \ref{pcons1} may not be true and the reason is that we might have solutions $\phi$ of the Euler-Lagrange equations (\ref{eulerlagrangeeq}) that are not solutions to some $\xi \in \mathcal{X}_{L}^{k}(T_{k}^{1}M)$.

However, under some assumption, the converse is true (\cite{modestobuc}):
\begin{proposition} \label{pcons2}
Let  $L$ be a regular Lagrangian on $T_k^1 M$ and assume that there exists a vector field $X \in \mathcal{X} (T_{k}^{1}M)$ such that
\begin{equation}
i_X (\omega_L)_A = d f_A, \forall 1 \leq A \leq k
\end{equation}
for some functions $f_A : T^1_k M \longrightarrow \mathbf{R}^{k}$.

Then $\mathbf{F}=(f_1, \ldots ,f_k)$ is a conservation law for the Euler-Lagrange equations (\ref{eulerlagrangeeq}) if and only if $\sum\limits_{A=1}^{k} \xi_A (f_A) = 0$, for all integrable SOPDE $\xi \in \mathcal{X}_{L}^{k}(T_{k}^{1}M)$.
\end{proposition}
Using the notions from the previous section, we have:
\begin{definition} \rm (\cite{modestobuc})
A vector field $X \in \mathcal{X}(T_k^1 M)$ is called a Cartan
symmetry for the regular Lagrangian $L$, if $L_X (\omega_L)_A =0$,
for all $1 \leq A \leq k$ and $L_X E_L =0$.
\end{definition}
In this case the flow $\phi_t$ of $X$ transforms solutions of the Euler-Lagrange equations on solutions of the Euler-Lagrange equations, that is, each $\phi_t$ is a symmetry of the Euler-Lagrange equations (\cite{modestobuc}, \cite{rrsv}).

Let us remark that  for a Cartan symmetry of $L$,  $L_X (\omega_L)_A =0$  implies that, locally, we
have $i_X (\omega_L)_A = d f_A$, for all $1 \leq A \leq k$ (\cite{modestobuc}). So, if $X$ is a Cartan symmetry, then proposition \ref{pcons2}  holds locally.
\begin{theorem}(Noether Theorem) (\cite{modestobuc}) \label{thmn1}
Let $L$ be  a regular Lagrangian on $T_k^1 M$ and  $X \in \mathcal{X}(T_k^1 M)$  a Cartan symmetry for $L$.
Then, there exists (locally defined) functions $f_A$ on $T_k^1 M$ such that
\begin{equation}
L_X (\theta_L)_A = df_A \; \, , 1 \leq A \leq k \, , \label{nt1}
\end{equation}
and  the following functions
\begin{equation}
\Phi_A =  (\theta_L)_A (X) - f_A \; \, , 1 \leq A \leq k \, ,
\label{nt2}
\end{equation}
give a conservation law for the Euler-Lagrange equations
associated to $L$, i.e. for any integrable evolution $k$-vector
field (defined locally) associated to $H=E_L$, the energy of $L$.
\end{theorem}
Next result show when a conservation law for a Lagrangian induces
and are induced by a Cartan symmetry.
\begin{theorem}(\cite{modestobuc}) Let be $L$ a regular Lagrangian on $T_k^1
M$, the functions $f_A \in C^{\infty} (T_k^1 M)$, $1 \leq A \leq
k$, and a vector field $X \in \mathcal{X}(T_k^1 M)$ such that
\begin{equation}
i_X (\omega_L)_A = d f_A \, \, , 1 \leq A \leq k \, . \label{nt3}
\end{equation}
Then $F=(f_1, \ldots , f_k)$ is a conservation law for $L$ if and
only if $X$ is a Cartan symmetry.
\end{theorem}
\begin{example}\rm  (\cite{modestobuc}, \cite{muntsalg}) The equation of motion of a vibrating string  is
\begin{equation}
\sigma \frac{\partial ^{2}\phi }{\partial (t^{1})^{2}}-\tau \frac{\partial
^{2}\phi }{\partial (t^{2})^{2}}=0,  \label{ex2}
\end{equation}
where $\sigma $ and $\tau $ are certain constants of the mechanical system.

The equation (\ref{ex2}) can be described as the
generalized Euler-Lagrange equations associated to a Lagrangian $L(x,v_{1},v_{2}) = \frac{1}{2}(\sigma v_{1}^{2}-\tau v_{2}^{2})$ defined
on the jet bundle $T_{k}^{1}M$ with $M=\mathbf{R}$ and $k=2$. Since $L$ is regular there exists a $k$-symplectic structure $((\omega
_{L})_{1},(\omega _{L})_{2})$, associated to $L$, given in local coordinates
by $(\omega _{L})_{1}=\sigma dv_{1}\wedge dx$, $(\omega _{L})_{1}=-\tau
dv_{2}\wedge dx$. The energy $E_{L}=\Delta(L)-L$ is locally given by $E_{L}=\frac{%
1}{2}(\sigma v_{1}^{2}-\tau v_{2}^{2})$ and $\ dE_{L}=\sigma
v_{1}dv_{1}-\tau v_{2}dv_{2}$.

Let us suppose that there exists $\xi = (\xi _{1},\xi _{2})  \in T_{2}^{1}(T_{2}^{1} \mathbf{R})$ a solution of the equation
\begin{equation}
i_{X_{1}}(\omega_{L})_{1}+i_{X_{2}}(\omega_{L})_{2}=dE_{L} \, .  \label{ex2p}
\end{equation}
Then,  $(\xi _{1},\xi _{2})$ is
a SOPDE and, locally, $\xi _{A}=v_{A}\frac{\partial }{\partial x}
+(\xi _{A})_{1}\frac{\partial }{\partial v_{1}}+(\xi _{A})_{2}\frac{\partial
}{\partial v_{2}} \,$, $A=1,2$. So, we have $\sigma (\xi _{1})_{1}-\tau (\xi
_{2})_{2}=0$ and  if we consider $\phi :\mathbf{R}^{2}\rightarrow \mathbf{R}$, $\phi =\phi (t^{1},t^{2})$, a solution of $\xi =(\xi _{1},\xi _{2})$, then
we obtain $ 0=\sigma (\xi _{1})_{1}-\tau (\xi _{2})_{2}=\sigma \frac{\partial ^{2}\phi }{
\partial (t^{1})^{2}}-\tau \frac{\partial ^{2}\phi }{\partial (t^{2})^{2}}$. Thus, the equation (\ref{ex2p}) is a geometric version for the equations (\ref{ex2}). An example of an integrable SOPDE solution $\xi =(\xi _{1},\xi
_{2})$ of (\ref{ex2}) is given by (\cite{modestobuc})
$$
\xi_1 = v_1 \frac{\partial}{\partial x}+ \tau (\sigma (v_1)^2 +
\tau (v_2)^2) \frac{\partial}{\partial v_1} + 2 \sigma \tau v_1
v_2 \frac{\partial}{\partial v_2} \, ,$$ $$ \xi_2 = v_2
\frac{\partial}{\partial x} + 2 \sigma \tau v_1 v_2
\frac{\partial}{\partial v_1} + \sigma (\sigma (v_1)^2 + \tau
(v_2)^2) \frac{\partial}{\partial v_2} \, .$$ Thus any solution
$\phi$ of the SOPDE $\xi =(\xi _{1},\xi _{2})$ in the formulae
above is a solution of the vibrating string equation (\ref{ex2}). The following two functions $\Phi_1, \Phi_2 : T_{2}^{1}\mathbf{R}\rightarrow \mathbf{R}$,
\begin{equation}
\Phi_1 (v_1, v_2) = -2 \sigma v_1 v_2 \, , \hspace{1cm} \Phi_2
(v_1, v_2) = \sigma (v_1)^2 + \tau (v_2)^2 \label{conslaw1}
\end{equation}
give a conservation law $\Phi = (\Phi_1, \Phi_2)$ for every
evolution $2$-vector field associated with the Hamiltonian $E_L$. We can say that $\Phi = (\Phi_1, \Phi_2)$ is a
conservation law for the Euler-Lagrange equations (\ref{ex2}) of
the vibrating string, or $\Phi = (\Phi_1, \Phi_2)$ give a
conservation law for an integrable evolution $k$-vector field
associated  to $H=E_L$. More that, this conservation law is not induced by
  a Cartan symmetry, and hence it will show that the converse of the Noether Theorem \ref{thmn1} is not true, unless
   the assumptions (\ref{nt3})  are satisfied (\cite{modestobuc}).
\end{example}
\section{New kinds of conservation laws for $k$-symplectic  systems}
In this section we will present a result which allow us to obtain
new kinds of conservation laws (nonclassical) for $k$-symplectic Hamiltonian and Lagrangian systems, without the help of a Noether type theorem and without
the use of a variational principle, using only symmetries and
pseudosymmetries associated to the $k$-vector fileds
$\mathbf{X}=(X_{1},\dots ,X_{k})$ which are solutions of the
equation $\sum\limits_{A=1}^{k}i_{X_{A}}\omega _{A}=dH$
(\cite{muntfliasi}). This result is a generalization from the
classical case of a results of G.L. Jones (\cite{jones})
and M. Cr\^{a}\c{s}m\u{a}reanu (\cite{crasmareanu}).

Let $\mathcal{M}$ be a manifold and $\mathbf{X}=(X_{1},\dots ,X_{k})$ be a $k$-vector field on $\mathcal{M}$.
\begin{definition} \rm The map $\mathbf{\Phi }=\left( \Phi _{1},\ldots ,\Phi _{k}\right)
: \mathcal{M} \longrightarrow \mathbf{R}^{k}$ is called {\em conservation
law} for  $\mathbf{X}=(X_{1},\dots ,X_{k})$   if
\begin{equation}
\sum\limits_{A=1}^{k} L_{X_{A}} \Phi_{A}=0.  \label{defcons}
\end{equation}
\end{definition}
\begin{definition} \rm
A vector field $Y$ on $\mathcal{M}$ is called {\em symmetry} or {\em dynamical symmetry} of the $k$-vector field $\mathbf{X}=(X_{1},\dots ,X_{k})$ if
\begin{equation}
L_{X_{A}}Y  = 0, \forall \quad A=1,\ldots ,k.  \label{defsym}
\end{equation}
\end{definition}
\begin{definition} \rm
A vector field $Y$ on $\mathcal{M}$ is called {\em pseudosymmetry} or {\em dynamical pseudosymmetry} of the $k$-vector field $\mathbf{X}=(X_{1},\dots ,X_{k})$ if,  for all $A=1,\ldots ,k$, there are functions $\lambda_A^B \in C^{\infty }(\mathcal{M})$, $B=1,\ldots , k$,  such that
\begin{equation}
L_{X_{A}}Y = \sum\limits_{B=1}^{k} \lambda_A^B X_B \, .   \label{defpseudosym}
\end{equation}
\end{definition}
This definitions is according to Krupkov\'{a} definition for symmetry of a distribution (\cite{krupkova}). Furthermore, we can give a generalization of this notion:
\begin{definition} \rm
If we fixed a $k$-vector field $\mathbf{Z}=(Z_{1},\dots ,Z_{k})$ on $\mathcal{M}$,
then a vector field $Y$ on $\mathcal{M}$ is called $\mathbf{Z}$-{\em pseudosymmetry} of the $k$-vector field $\mathbf{X}=(X_{1},\dots ,X_{k})$ if, for
all $A=1,\ldots ,k$, there are functions $\lambda_A^B \in C^{\infty }(\mathcal{M})$, $B=1,\ldots , k$,  such that
\begin{equation}
L_{X_{A}}Y = \sum\limits_{B=1}^{k} \lambda_A^B Z_B \, .   \label{defZpseudosym}
\end{equation}
\end{definition}
Obviously, a $X$-pseudosymmetry of $\mathbf{X}=(X_{1},\dots ,X_{k})$ is a pseudosymmetry of $\mathbf{X}$ and a $\mathbf{O}$-pseudosymmetry of  $\mathbf{X}$ is a symmetry for $\mathbf{X}$.

Next, we will present the main result of the paper, which allow us
to obtain new kinds of conservation laws for $k$-symplectic
Hamiltonian and Lagrangian systems, without the help of a Noether type theorem
and without the use of a variational principle.
\begin{theorem}
Let be $\mathbf{X}=(X_{1},\dots ,X_{k})$ a $k$-vector field on $\mathcal{M}$
and $( \omega _{1}, \ldots , \omega _{k} )$ be a family of
$p$-forms on $\mathcal{M}$, invariant for $\mathbf{X}$, i.e. $L_{X_{A}}\omega _{A}=0$, for all $%
A=1,\ldots ,k$. Let $Y$ be a symmetry of $\mathbf{X}$ and the $k$-vector field $\mathbf{Y}=(Y, \ldots , Y)$. If we have $p-1$ vector fields on $\mathcal{M}$, $S_1$, ..., $S_{p-1}$, which are $\mathbf{Y}$-pseudosymmetries of $\mathbf{X}$, then
\begin{equation}
\mathbf{\Phi }=\left( \Phi _{1},\ldots ,\Phi _{k}\right) ,
\label{lcons}
\end{equation}
is a conservation law for $\mathbf{X}=(X_{1},\dots ,X_{k})$,
where $\Phi _{A}=\omega _{A}\left( S_{1},\ldots
,S_{p-1},Y \right) $, $A=1,\ldots ,k \,$,
or locally,
\begin{equation}
\Phi _{A}=\left(
S_{1}\right) ^{i_{1}}\cdots \left( S_{p-1}\right)
^{i_{p-1}}\left( Y \right) ^{i_{p}} \left( \omega_{A} \right)_{i_{1}\ldots
i_{p-1}i_{p}} \, .
\end{equation}
Particularly, if $Y, S_{1}$, ...,
$S_{p-1}$ are symmetries for $\mathbf{X}$ then
$\mathbf{\Phi }$ given by (\ref{lcons}) is a conservation law
for $\mathbf{X}$.
\end{theorem}
\begin{proof} Applying the properties of the Lie derivative
and taking into account that, locally, $\Phi _{A}=\left(
S_{1}\right) ^{i_{1}}\cdots \left( S_{p-1}\right)
^{i_{p-1}}\left( Y \right) ^{i_{p}} \left( \omega_{A} \right)_{i_{1}\ldots
i_{p-1}i_{p}}$, for all $A=1,\ldots ,k$, we obtain that

$\sum\limits_{A=1}^{k}L_{X_{A}}\Phi
_{A}=\sum\limits_{A=1}^{k}L_{X_{A}}\left( \left(
S_{1}\right) ^{i_{1}}\cdots \left( S_{p-1}\right)
^{i_{p-1}}\left( Y \right)^{i_{p}} \left( \omega_{A} \right)_{i_{1}\ldots
i_{p-1}i_{p}} \right) =$

$=\sum\limits_{A=1}^{k}\left( L_{X_{A}}S_{1}\right)
^{i_{1}}\left( S_{2}\right) ^{i_{2}}\cdots \left(
S_{p-1}\right) ^{i_{p-1}}\left( Y \right)^{i_{p}}\left(
\omega _{A}\right) _{i_{1}\ldots i_{p-1}i_{p}}+\cdots +$

$+\sum\limits_{A=1}^{k}\left( S_{1}\right) ^{i_{1}}\cdots
\left( S_{p-2}\right)^{i_{p-2}}\left(
L_{X_{A}}S_{p-1}\right) ^{i_{p-1}}\left( Y \right)
^{i_{p}}\left( \omega _{A}\right) _{i_{1}\ldots i_{p-1}i_{p}}+$

$+\sum\limits_{A=1}^{k}\left( S_{1}\right) ^{i_{1}}\cdots
\left( S_{p-1}\right) ^{i_{p-1}}\left( L_{X_{A}}Y \right)
^{i_{p}}\left( \omega _{A}\right) _{i_{1}\ldots i_{p-1}i_{p}}+$

$+\sum\limits_{A=1}^{k}\left( S_{1}\right) ^{i_{1}}\cdots
\left( S_{p-1}\right) ^{i_{p-1}}\left( Y \right)
^{i_{p}}\left( L_{X_{A}}\omega _{A}\right) _{i_{1}\ldots
i_{p-1}i_{p}}=$

$=\sum\limits_{A=1}^{k}  \left( \sum\limits_{B=1}^{k} (\lambda_1)_A^B \right) \left( Y \right) ^{i_{1}}\left(
S_{2}\right) ^{i_{2}}\cdots \left( S_{p-1}\right)
^{i_{p-1}}\left( Y  \right)^{i_{p}}\left( \omega _{A} \right)
_{i_{1}\ldots i_{p-1}i_{p}}+\cdots $

$+\sum\limits_{A=1}^{k}\left( S_{1}\right) ^{i_{1}}\cdots
\left( S_{p-2}\right) ^{i_{p-2}} \left( \sum\limits_{B=1}^{k} (\lambda_{p-1})_A^B \right) \left( Y \right)
^{i_{p-1}}\left(
Y \right) ^{i_{p}}\left( \omega _{A}\right) _{i_{1}\ldots i_{p-1}i_{p}}=0$,\\ because $L_{X_{A}}Y =0$, $L_{X_{A}}\omega_{A}=0$ and taking
into account of the antisymmetry of $\omega _{A}$.
\end{proof}
As an immediate consequence of the previous theorem, we have the
result:
\begin{theorem}
\bigskip Let $(\mathcal{M},\omega _{A},V;1\leq A\leq k)$ be a $k$-symplectic manifold
and $H:\mathcal{M}\longrightarrow \mathbb{R}$ be a function on $\mathcal{M}$. Let $\mathbf{X}
=(X_{1},\dots ,X_{k})$ be an integrable evolution $k$-vector field
associated to $H$, i.e. $\mathbf{X} \in \mathcal{X}_{H}^{k}( \mathcal{M} )$. Let us suppose that $L_{X_{A}}\omega _{A}=0$, for all $A=1,\ldots ,k$. Let $Y$ be a symmetry of $\mathbf{X}$ and the $k$-vector field $\mathbf{Y}=(Y, \ldots , Y)$. If we have a vector field $S$ on $\mathcal{M}$  which is a $\mathbf{Y}$-pseudosymmetry of $\mathbf{X}$, then
\begin{equation*}
\mathbf{\Phi }=\left( \Phi _{1},\ldots ,\Phi _{k}\right) ,
\end{equation*}
is a conservation law for $\mathbf{X}=(X_{1},\dots ,X_{k})$, where
$\Phi_{A} = \omega_{A}\left( S, Y \right) $, for all $A=1,\ldots ,k$.

Particularly, if $Y$, $S$ are symmetries for $\mathbf{X}$ then $\mathbf{\Phi }$ is a conservation law for $\mathbf{X}=(X_{1},\dots ,X_{k})$.
\end{theorem}

\begin{remark}\rm
a) Obviously, for any $k$-vector field $\mathbf{X} \in \mathcal{X}_{H}^{k}(\mathcal{M})$, using (\ref{hdwplus}), we have
$\sum\limits_{A=1}^{k}\,L_{X_{A}}\omega _{A}=0$. But, for our purpose we need  $L_{X_{A}}\omega _{A}=0$, $A=1,\ldots ,k$.

b) The Hamiltonian function $H$ is not a conservation law for an
integrable evolution $k$-vector field $\mathbf{X}=(X_{1},\dots ,X_{k}) \in \mathcal{X}_{H}^{k}(\mathcal{M})$. Neither the map
$\mathbf{H}=\left( H,\ldots ,H\right) :\mathcal{M} \longrightarrow
\mathbf{R}^{k}$ is not a conservation
law for any integrable evolution $k$-vector field $\mathbf{X} \in \mathcal{X}_{H}^{k}(\mathcal{M})$.
\end{remark}
Now, using this last result we can obtain new kinds of
conservation laws for $k$-symplectic Hamiltonian systems and
$k$-symplectic Lagrangian systems.
\begin{corollary}
\bigskip Let $\left( (T_{k}^{1})^{\ast }M,(\omega _{0})_{A},H\right) $\ be a
$k$-symplectic Hamiltonian system and $\mathbf{X}=(X_{1},\dots
,X_{k})$ be
an integrable evolution $k$-vector field associated to $H$, i.e. $\mathbf{
X} \in \mathcal{X}_{H}^{k}((T_{k}^{1})^{\ast }M)$. Let us suppose that $L_{X_{A}}(\omega _{0})_{A}=0$, for all $A=1,\ldots ,k$. Let $Y \in \mathcal{X} ((T_{k}^{1})^{\ast }M)$ be a symmetry of $\mathbf{X}$ and the $k$-vector field $\mathbf{Y}=(Y, \ldots , Y)$ on $(T_{k}^{1})^{\ast }M$. If we have a vector field $S$ on $(T_{k}^{1})^{\ast }M$  which is a $\mathbf{Y}$-pseudosymmetry of $\mathbf{X}$, then
\begin{equation*}
\mathbf{\Phi }=\left( \Phi _{1},\ldots ,\Phi _{k}\right) ,
\end{equation*}%
is a conservation law for $\mathbf{X}=(X_{1},\dots ,X_{k})$, where
$\Phi _{A}=(\omega _{0})_{A}\left( S, Y \right) $, for all  $A=1,\ldots ,k$.

Particularly, if $\mathbf{Y}$, $\mathbf{S}$ are symmetries for
$\mathbf{X}$
then $\mathbf{\Phi }$  is a conservation law for $\mathbf{X}=(X_{1},\dots ,X_{k})$.
\end{corollary}
\begin{corollary}
\bigskip Let $\left( T_{k}^{1}M,(\omega _{L})_{A},E_{L}\right) $ be a $k$-symplectic Lagrangian system and $\mathbf{X}=(X_{1},\dots
,X_{k})$ be an integrable evolution $k$-vector field associated to
$H=E_L$, i.e. $\mathbf{X} \in \mathcal{X}_{L}^{k}(T_{k}^{1}M)$. Let us suppose that $L_{X_{A}}(\omega
_{L})_{A}=0$, for all $A=1,\ldots ,k$. Let $Y \in \mathcal{X} (T_{k}^{1}M)$ be a symmetry of $\mathbf{X}$ and the $k$-vector field $\mathbf{Y}=(Y, \ldots , Y)$ on $T_{k}^{1}M$. If we have a vector field $S$ on $T_{k}^{1}M$  which is a $\mathbf{Y}$-pseudosymmetry of $\mathbf{X}$, then
\begin{equation*}
\mathbf{\Phi }=\left( \Phi _{1},\ldots ,\Phi _{k}\right) ,
\end{equation*}
is a conservation law for $\mathbf{X}=(X_{1},\dots ,X_{k})$, where
$\Phi _{A}=(\omega _{L})_{A}\left( S , Y \right) $, for all
$A=1,\ldots ,k$.

Particularly, if $\mathbf{Y}$, $\mathbf{S}$ are symmetries for
$\mathbf{X}$
then $\mathbf{\Phi }$ is a  conservation law for $\mathbf{X}=(X_{1},\dots ,X_{k})$.
\end{corollary}
\begin{remark}\rm
If each vector fields $X_1$, $\ldots$, $X_k$  of
$\mathbf{X} \in \mathcal{X}_{L}^{k}(T_{k}^{1}M)$ are Cartan
symmetries for $L$, then we have $L_{X_{A}}(\omega _{L})_{A}=0$,
for all $A=1,\ldots , k$, and then we can apply  the last
corollary for this $k$-vector field $\mathbf{X}$. Moreover, we
have that $(H,\ldots ,H)$ is a conservation law for
$\mathbf{X}=(X_{1},\dots ,X_{k})$, where $H=E_L$.
\end{remark}
\begin{example}\rm  (\cite{modestobuc}, \cite{olver})
a) If we consider the Lagrangians $L_1, L_2: T_2^1 \mathbf{R} \rightarrow \mathbf{R}$, $$L_1(x,v_1,v_2)=\frac{1}{2} ( \sigma (v_1)^2 - \tau (v_2)^2  ) \, , \, \, L_2(x,v_1,v_2)= \sqrt{1+(v_1)^2 + (v_2)^2} \, , $$ then the vector field $X=\frac{\partial}{\partial x}$ is a Cartan symmetry for $L_1$ and $L_2$. The induced conservation laws are $\mathbf{\Phi}=( \Phi_{1}= \sigma v_1, \Phi_2= -\tau v_2)$ for $L_1$ and\\$\mathbf{\Phi}=( \Phi_{1}= \frac{v_1}{\sqrt{1+(v_1)^2 + (v_2)^2}}, \Phi_2= \frac{v_2}{\sqrt{1+(v_1)^2 + (v_2)^2}} )$ for $L_2$.\\
Let us observe that the above Lagrangians corresponde to the
vibrating string equations and, respectively to the equations of
minimal surfaces.

b) For the Lagrangian $L: T_3^1 \mathbf{R} \rightarrow
\mathbf{R}$, defined by
$$L(x,v_1,v_2,v_3)=\frac{1}{2}\left(
(v_1)^2 + (v_2)^2 +(v_3)^2  \right) \, ,$$
the vector field $X=\frac{\partial}{\partial x}$ is a Cartan symmetry, and the
induced conservation law is $\mathbf{\Phi }=\left( \Phi_{1},\Phi_2
,\Phi_{3}\right)$, where $\Phi_i  = v_i$, $i=1,2,3$. The
Euler-Lagrange equations corresponding to $L$ are the Laplace's
equations.

c) For the Lagrangian $L: T_2^1 \mathbf{R}^2 \rightarrow
\mathbf{R}$, defined by
$$L(x^1,x^2,v^1_1,v^1_2,v_1^2,v_2^2) = \left( \frac{1}{2}\lambda +\nu \right)\left[ (v_1^1)^2+(v_2^2)^2 \right]
 +\frac{1}{2}\nu  \left[ (v_2^1)^2+(v_1^2)^2 \right] +(\lambda+\nu) v_1^1 v_2^2 \, ,$$
the vector field $X=\frac{\partial}{\partial
x^1}+\frac{\partial}{\partial x^2}$ is a Cartan symmetry, and the
induced conservation law is $\mathbf{\Phi }=\left( \Phi_{1},\Phi_2
\right)$, where
 $\Phi_1  = (\lambda+2\nu)v_1^1+\nu v_1^2 +(\lambda+\nu)v_2^2$, $\Phi_2  = (\lambda+\nu)v_1^1+\nu v_2^1 +(\lambda+2\nu)v_2^2$.
The Euler-Lagrange equations corresponding to $L$ are the Navier's
equations.
\end{example}

\textbf{Acknowledgments.} This research was supported by FP7-PEOPLE-2012-IRSES-316338.

\end{document}